\documentclass[11pt,a4paper]{amsart}

\usepackage{fullpage}
\usepackage{amsmath}
\usepackage{amsfonts}
\usepackage{amssymb}
\usepackage{amsthm}
\usepackage[utf8]{inputenc}
\usepackage{breqn}
\usepackage{afterpage}
\usepackage{longtable}
\usepackage{indentfirst}
\usepackage{caption}
\usepackage{subcaption}
\usepackage{graphicx}
\graphicspath{ {./images/} }
\newtheorem{theorem}{Theorem}[section]
\newtheorem{lemma}[theorem]{Lemma}

\newtheorem{proposition}[theorem]{Proposition}

\theoremstyle{definition}
\newtheorem{definition}[theorem]{Definition}

\newtheorem{theorem-definition}[theorem]{Theorem-Definition}

\theoremstyle{remark}
\newtheorem{remark}[theorem]{Remark}

\numberwithin{equation}{section}

\usepackage{color}
\usepackage{xcolor}

\def\red{\color{red}}

\usepackage{bbold}

\newcommand{\lam}{\lambda}

\newcommand{\C}{\mathbb{C}}
\newcommand{\R}{\mathbb{R}}
\newcommand{\N}{\mathbb{N}}
\newcommand{\Z}{\mathbb{Z}}
\renewcommand{\P}{\mathbb{P}}

\renewcommand{\L}{\mathcal{L}}
\newcommand{\J}{\mathcal{J}}
\newcommand{\F}{\mathcal{F}}
\newcommand{\M}{\mathcal{M}}
\newcommand{\G}{\mathcal{G}}

\DeclareMathOperator{\Card}{Card}

\DeclareMathOperator{\Supp}{Supp}
\DeclareMathOperator{\diam}{diam}

\DeclareMathOperator{\Lip}{Lip}
\DeclareMathOperator{\dist}{dist}

\renewcommand{\epsilon}{\varepsilon}

\begin{document}

\title{Holomorphic motions of weighted periodic points}

\author{Fabrizio Bianchi}
\author{Maxence Br\'evard}

\address{CNRS, Univ. Lille, UMR 8524 - Laboratoire Paul Painleve, F-59000 Lille, France}
\email{fabrizio.bianchi@univ-lille.fr}
\address{Université de Toulouse - IMT, UMR CNRS 5219, 31062 Toulouse Cedex, France}
\email{maxence.brevard@univ-toulouse.fr, mbrevard@hotmail.fr}



\subjclass[2010]{}
\date{}


\begin{abstract}
We study the holomorphic motions of repelling
periodic points
in stable families of endomorphisms of $\mathbb P^k(\mathbb C)$.
In particular, we establish an asymptotic equidistribution
of the graphs associated to such periodic points with respect to natural measures in the space of all holomorphic motions of points in the Julia sets.
\end{abstract}

\maketitle


\section{Introduction}

A \emph{holomorphic family of endomorphims of $\mathbb P^k$} is a pair $(M,f)$, where $M$ is a complex manifold and $f\colon M\times \mathbb P^k\to M\times \mathbb P^k$ is a holomorphic map of the form $f(\lambda,z)=(\lambda, f_\lambda (z))$, where each $f_\lambda$ is an endomorphism of $\mathbb P^k$ of the same algebraic degree $d$. 
We always assume 
that $M$ is connected and simply connected, and that $d\geq 2$.
The following fundamental result due 
by Lyubich \cite{Ly83a}, Mañ\'e-Sad-Sullivan \cite{MSS83}, and DeMarco \cite{dM03} defines and 
characterizes \emph{stability} within such families when $k=1$, see also
\cite{Lev82,Prz85,Si81} 
for further characterizations and previous results in the polynomial case. Recall that
Freire-Lopes-Mañ\'e \cite{FLM} and Lyubich \cite{Ly83b} proved that each rational map $f_\lam$ admits a unique invariant measure of maximal entropy $\mu_\lam$,
whose support, denoted as $J_\lam$, is the Julia set of $f_\lam$.

\begin{theorem-definition}
Let $(M,f)$ be a holomorphic family of rational maps as above.
The following conditions are equivalent:
\begin{enumerate}
\item the Julia sets $J_\lam$ move holomorphically with $\lam$;
\item $dd^c L(\lam)\equiv 0$, where $L(\lam):= \int \log |f'_\lambda|\mu_\lam$
is the Lyapunov exponent of the measure of maximal entropy $\mu_\lambda$ of $f_\lambda$;
\item the repelling periodic points of $f_\lam$ move holomorphically with $\lam$.
\end{enumerate}
We say that the family is \emph{stable} if any (hence, all) of the above conditions hold. 
\end{theorem-definition}

Recall that a family of Borel sets $E_\lam \subset \P^1$ \emph{move holomorphically with $\lam$} if
there exists a set $\L$ of holomorphic functions $\gamma : M \to \P^1$ such that
the graphs $\Gamma_{\gamma_1}, \Gamma_{\gamma_2}\subset M\times \mathbb P^1$
of two distinct functions $\gamma_1, \gamma_2 \in \L$ do not intersect,
and for any parameter $\lam \in M$, we have $\L_\lam := \{\gamma(\lam), \gamma \in \L \} = E_\lam$.
A crucial point here is that, if the sets $E_\lam$ move holomorphically with $\lam$, the same is true for $\overline {E_\lam}$. This fact, usually referred to as the $\lambda$--lemma, is a consequence of the Hurwitz theorem for (one-dimensional) holomorphic maps.

\medskip

A generalization of the above result for families in any dimension $k\geq 1$ was
proved by Berteloot, Dupont, and the first author in \cite{BBD18}.
As the Hurwitz theorem fails in higher dimensions,
the approach replaces the holomorphic motion of the Julia sets by a \emph{measurable holomorphic motion},
namely a family $\L$ of non-intersecting graphs $\gamma \colon M \to \P^k$ such that for any parameter $\lam \in M$, we have $\mu_\lam(\L_\lam) = 1$,
where $\mu_\lam$ is the unique  measure of
maximal entropy of the system $(\P^k, f_\lam)$ (see 
\cite{BD01,DS10,FS94}).
We still denote its support as $J_\lambda$. The result is as follows, see also
\cite{BBD18,BB22,B19} for details and further characterizations and
\cite{BB18}
for an explanation of the strategy of the proof.

\begin{theorem-definition}\label{t:bbd}
Let $(M,f)$ be a holomorphic family of endomorphisms of $\mathbb P^k$. 
The following conditions are equivalent:
\begin{enumerate}
\item there exists a measurable holomorphic motion for the Julia sets $J_\lam$;
\item $dd^c L(\lam)\equiv 0$, 
where $L(\lam):=\int \log |Df_\lam|\mu_\lambda$
is the sum of the  Lyapunov exponents of the measure of maximal entropy $\mu_\lambda$ of $f_\lambda$.
\end{enumerate}
We say that the family is \emph{stable} if any (hence, all) of the above conditions hold. 

Moreover, the following condition
\begin{itemize}
\item[(3)] the repelling periodic points of $f_\lam$ move holomorphically with $\lam$
\end{itemize}
implies the above two, and is equivalent to them if $k=2$, or if $M$ is an open connected and simply connected
subset of 
 the space 
 $\mathcal H_d(\P^k)$ of all holomorphic endomorphisms of $\P^k$ of algebraic degree $d$.
 \end{theorem-definition}

It is still an open question whether the stability of 
a general family of endomorphisms in any dimension is equivalent to the motion of all the repelling periodic cycles. 
On the other hand, it was proved in \cite{B16,B19} that the stability of a family implies at least
a weaker version of the motion of the repelling cycles. Notice that, in turn,
this weaker notion is still sufficient
to imply stability, see also \cite{B18}.
The definition is as follows.

 \begin{definition}\label{def:motion-rep-asympt}
  We say that \emph{asymptotically all} repelling cycles move holomorphically  on $M$
  if
  there exists a countable set
  $\mathcal P= \cup_{n\in \N^*} \mathcal P_n$ of holomorphic functions $\gamma : M \to \P^k$
  satisfying the following properties:
 \begin{enumerate}
 \item  $\mathrm{Card }\: \mathcal P_n = d^{kn} + o(d^{kn})$;
 \item the point $\gamma(\lambda)$
 belongs to $J_\lambda$ and is a fixed point of $f_\lambda^n$, for every $\lambda \in M$ and $\gamma \in \mathcal P_n$;
 \item  for all open subsets $M'\Subset M$, we have
\[ \frac{\mathrm{Card } \{
 \gamma \in \mathcal P_n \: \colon \: \gamma(\lambda) \mbox{ is repelling for all } \lambda \in M'\} }{d^{kn}}\underset{n\to \infty}{\longrightarrow} 1.\]
\end{enumerate}
\end{definition}


The cycles in Definition \ref{def:motion-rep-asympt}
can be seen, in some sense, as generic with respect to the measure of maximal entropy.
More precisely, consider the space $\mathcal J$, defined as
\begin{equation}\label{e:J}
\mathcal J := \{
\gamma : M\to \mathbb P^k \: \colon \: \gamma(\lam) \in J_\lam \quad \forall \lam \in M
\}, \end{equation}
where the maps $\gamma$ are holomorphic. 
$\mathcal J$ is a metric space, see Section \ref{s:prelim-inverses} for details.
We can turn $\mathcal J$ into a topological dynamical system by defining a natural map $\mathcal F$ as
\begin{equation}\label{e:F}
\mathcal F (\gamma) (\lambda) := f_\lambda (\gamma(\lambda)).
\end{equation}
A
\emph{web} is a probability measure
$\mathcal M$
compactly supported on $\mathcal J$ which is invariant under $\mathcal F$. An \emph{equilibrium web} 
\cite{BBD18}
is a web such that
\[
(p_\lam)_* \mathcal M = \mu_\lam \quad  \forall \lambda \in M,
\]
where we denote by $p_\lam\colon \mathcal J \to \mathbb P^k$ the natural map $\gamma \mapsto \gamma(\lambda)$. The properties of equilibrium webs are crucial in the approach in \cite{BBD18}.

\medskip

With this terminology, and owing to the equidistribution
of the repelling periodic cycles with respect to the measure of maximal entropy \cite{BD99},
we see in particular that
the fact that the existence of $\mathcal P$ as in Definition \ref{def:motion-rep-asympt} is equivalent to stability as in Theorem-Definition \ref{t:bbd}
leads to the following equidistribution result for periodic graphs in any stable family. This can be seen as a weak version of the implication (1)$\Rightarrow$(3).

\begin{theorem}\label{t:graph_equidistribution}
Let $(M,f)$ be a stable family of
endomorphisms of $\mathbb P^k$.
Then, for every 
$M'\Subset M$ and every
$n\in \mathbb N^*$, there exists
a non-empty subset $\mathcal P_n \subset \J$ of motions $\gamma$
of $n$-periodic points
such that 
$\gamma(\lam)$ is repelling for all $\lam \in M'$ and
\[
\lim_{n\to \infty} d^{-kn} \sum_{\gamma \in \mathcal P_n}
\delta_{\gamma} = \mathcal M,
\]
where $\mathcal M$ is an equilibrium web.
 \end{theorem}



In this paper, we address the question of the motion of cycles which similarly equidistribute invariant measures in  a much larger class, that we now introduce.

\medskip

Let $f$ be an endomorphism of $\mathbb P^k$ 
of algebraic degree $d\geq 2$
and let $\phi$ be a real continuous function on $\mathbb P^k$ (usually called a \emph{weight}).
The \emph{pressure} of $\phi$ is defined as 
\[
P(\phi) := \sup_{\nu} 
\big(h_\nu + \int \phi \nu\big)
\]
where 
the supremum is over all $f$-invariant measures $\nu$
and
$h_\nu$ denotes the measure-theoretic entropy of $\nu$. An 
\emph{equilibrium state} 
for the weight $\phi$ is defined  as a maximizer of the pressure function, i.e., as an invariant measure $\mu_\phi$ satisfying
\[
h_{\mu_\phi} + \int \phi \mu =  P(\phi).
\]
Assume that $f$ satisfies the following condition:

\medskip\noindent
{\bf (A)} \hspace{1cm} the local degree of the iterate $f^n$ satisfies
$$\lim_{n\to\infty} \frac{1}{n} \log\max_{a\in\P^k}\deg(f^n,a) =0.$$

\medskip\noindent
Here, $\deg(f^n,a)$ is the multiplicity of $a$ as a solution of the
equation $f^n(z)=f^n(a)$. Note that generic endomorphisms of $\P^k$ satisfy
 this condition, 
 see \cite{DS10b}.
 Let $\phi$ satisfy

\medskip\noindent
{\bf (B)}
 \hspace{2cm}
$\|\phi\|_{\log^q} 
<\infty    
\, \mbox{ for some } q>2 \quad \quad \mbox{ and } \quad \quad 
\Omega (\phi)
<\log d$,

\medskip\noindent
where we define $\Omega(\phi):=\max (\phi)-\min (\phi) $,
\[\|\phi\|_{\log^q} 
:= 
\sup_{a,b \in \P^k} |\phi(a)-\phi(b)|\cdot (\log^\star \dist (a,b) )^q,\] 
and
$\log^\star (\cdot) = 1+|\log (\cdot)|$.
The existence and uniqueness of the equilibrium state $\mu_\phi$  for $f$
under the
assumptions
{\bf (A)} and {\bf (B)}
have been proved 
in \cite{UZ13,BD23}
see also \cite{D12,SUZ,BD22}
for further properties of these measures, and
\cite{PU} and references therein
for previous results in dimension 1. 
The case $\phi=0$ corresponds to the case of the measure of maximal entropy, see for instance
\cite{DS10} and references therein
for an account of this case.

\medskip

By the definition of pressure and the assumption on $\Omega(\phi)$, all these equilibrium states satisfy $h_{\mu_\phi}> \log d^{k-1}$. It is proved in \cite{BR22}
that, given a stable family 
$(f_\lam)_{\lam\in M}$
and any $\lam_0\in M$,
for any $f_{\lambda_0}$-invariant  measure $\nu$ 
satisfying $h_\nu >\log d^{k-1}$, it is possible to construct an associated web $\mathcal M_{\lam_0, \nu}$
with the property that
\[
(p_{\lambda_0})_* \mathcal M_{\lam_0, \nu} = \nu,
\]
as well as the associated lamination.
This in particular applies to the equilibrium states as above. 
The following is our main result.


\begin{theorem}\label{t:main}
Let $(M,f)$ be a stable family of
endomorphisms of $\mathbb P^k$.
Take $\lam_0 \in M$ and assume that $f_{\lam_0}$ satisfies condition 
{\bf (A)}. 
Let $\phi : \mathbb P^k\to \mathbb R$
 satisfy 
 {\bf (B)}
and let $\mu_\phi$ be the equilibrium state for $f_{\lam_0}$
 associated to $\phi$. 
Then, for every 
$M'\Subset M$ and every
$n\in \mathbb N$, there exists
a non-empty subset $\mathcal P_{\phi,n} \subset \J$ of motions 
$\gamma$ of $n$-periodic points
such that 
$\gamma(\lam)$ is repelling for all $\lam \in M'$ and
\[
\lim_{n\to \infty} e^{- n P(\phi)} \sum_{\gamma \in \mathcal P_{\phi,n}}
e^{\phi (\gamma(\lambda_0)) + \dots + \phi ( f^{n-1}_{\lambda_0} (\gamma(\lambda_0)))} \delta_{\gamma} = \mathcal M_{\lam_0,\mu_\phi}.
\]
 \end{theorem}

Observe that, in particular, Theorem \ref{t:main} generalizes
to general weights $\phi$
Theorem \ref{t:graph_equidistribution} ,
the latter corresponding to the case $\phi = 0$.


\medskip

At the parameter $\lam_0$, the equidistribution of repelling periodic points with respect to the equilibrium state $\mu_\phi$ has been established in \cite[Theorem 4.10]{BD22}. The proof follows the now classical strategy by Briend-Duval \cite{BD99},
who showed this result for the measure of maximal entropy, which corresponds to the case $\phi=0$. On the other hand, when $\phi\neq 0$, as the Jacobian of $\mu_\phi$ is not constant, the proof requires more precise estimates on the contraction along generic inverse branches for $\mu_\phi$, which in turn follow from delicate distortion estimates along inverse branches due to Berteloot-Dupont-Molino \cite{BD19,BDM}.
In the current paper, we adapt this strategy
in the setting of the dynamical system
$(\mathcal J, \mathcal F, \mathcal M_{\lam_0, \phi})$. This requires to precisely control the contraction of $f$ on tubes 
(i.e., tubular neighbourhoods, of uniform radius in $\lambda$, 
of the graphs
in $M'\times \mathbb \P^k$ of elements of $\mathcal J$) centered at $\mathcal M_{\lam_0, \phi}$-generic elements of $\mathcal J$. As a result, we get the motions of repelling periodic points as repelling periodic elements for $\mathcal F$.

\subsection*{Acknowledgments}
The first author would like to thank the Simons Foundation, Laura De Marco, and Mattias Jonsson, for supporting and organizing the Simons Symposium 
on Algebraic, Complex, and Arithmetic Dynamics
in August 2022. This work was motivated by questions and discussions which arised during such event.

This project has received funding from
 the French government through the Programme
 Investissement d'Avenir
 (I-SITE ULNE /ANR-16-IDEX-0004,
 LabEx CEMPI /ANR-11-LABX-0007-01,
ANR QuaSiDy /ANR-21-CE40-0016,
ANR PADAWAN /ANR-21-CE40-0012-01)
managed by the Agence Nationale de la Recherche.

\section{Nice inverse branches for expanding webs}\label{s:prelim-inverses}

\subsection{Transfer operators on $(\mathcal J, \mathcal F)$.}




\medskip


We consider 
in this section a holomorphic family $(M,f)$
of endomorphisms of $\mathbb P^k$.
We let $\mathcal J$ be as in \eqref{e:J}
(observe that stability is not required to define this set).
We can turn $\mathcal J$ into a Polish (i.e.,  separable complete metric) space $(\J, \dist_\J)$ as a closed subset of the space $\mathcal O(M,\P^k)$ endowed with the metric of local uniform convergence.
More precisely,
the distance between two elements $\gamma_1, \gamma_2 \in \J$ is given by
\[
\dist_\J(\gamma_1,\gamma_2) := \sum_{n=0}^{+\infty} 2^{-n} \max\left(1, \sup_{\lam \in K_n}\dist_{\P^k}(\gamma_1(\lam),\gamma_2(\lam))\right),
\]
where the family $(K_n)_{n\in \N}$ is an exhaustion of $M$, 
namely a nested sequence of compact sets whose union is $M$ and such that $K_n$ is a subset of the interior of $K_{n+1}$ for every $n\in \N$.
The resulting topology is independent of the choice of the exhaustion.

Let  
$\mathcal F\colon \mathcal J\to \mathcal J$ be as in \eqref{e:F}. 
For $\gamma \in \mathcal J$, we denote by $\Gamma_\gamma$ the graph of $\gamma$ in $M\times \mathbb P^k$.
We denote by $\mathcal J_s$ the subset of $\mathcal J$ given by
$$\mathcal J_s := \{\gamma\in\mathcal{J}:\Gamma_\gamma\cap GO(C_f)
  \ne \emptyset\},$$
and set $\mathcal X := \mathcal J \setminus \mathcal J_s$.
Observe that every element $\gamma \in\mathcal X$ admits $d^{k}$ well defined inverse elements by
$\mathcal F$ in $\mathcal X$, i.e.,
elements $\gamma'\in \mathcal X$ such that $\mathcal F (\gamma')=\gamma$.

We let $\psi\colon \mathcal J\to \R$ be a continuous function 
and define the (transfer) operator $\Lambda_\psi$ acting on measurable real functions on $\mathcal J$ as
%
\begin{equation}\label{e:def-Lambda-psi}
\Lambda_\psi (g)(\gamma)
=\sum_{\mathcal F(\gamma')=\gamma} e^{\psi (\gamma')} g (\gamma').
\end{equation}
Observe that the operator $\Lambda_\psi$ preserves positivity. However, even when $g$ is continuous, 
$\Lambda_\psi(g)$ needs not be continuous
(as, for example, the system $\mathcal F\colon \mathcal J\to \mathcal J$ may not have a well defined degree). On the other hand, as every $\gamma\in\mathcal X$ has precisely
$d^{k}$ preimages under $\mathcal F$, which are also in $\mathcal X$, the operator $\Lambda_\psi$ defines a continuous operator from 
$\mathcal C^0(\mathcal X)$
to
$\mathcal C^0(\mathcal X)$.

Given a positive measure $\mathcal N$ satisfying $\mathcal N (\mathcal J_s)=0$,
the measure $\mathcal N$  integrates any continuous functions on $\mathcal X$, and we can define $\Lambda_\psi^*\mathcal N$ by the relation
\[
\langle \Lambda_\psi^* \mathcal N, g\rangle=
\langle \mathcal N , \Lambda_\psi (g)
\rangle,
\]
where $g$ is any continuous function on $\mathcal J$.
We will use
the operator 
$\Lambda_\psi^*$
only on measures vanishing on $\mathcal J_s$. 

\begin{lemma}\label{l:rho-N-equivalent}
The following assertions are equivalent:
\begin{enumerate}
\item there exists a continuous and strictly positive
function $\theta\colon \mathcal X \to \mathbb R$ such that $\Lambda_\psi^n (g)\to c_g \theta$ for every continuous function $g\colon \mathcal J \to \mathbb R$, where the constant $c_g$ depends linearly and continuously on $g$;
\item
there exists a positive 
measure $\mathcal N$
on $\mathcal J$, satisfying $\mathcal N(\mathcal J_s)=0$, such that $(\Lambda^*_\psi)^n \delta_\gamma\to c_\gamma \mathcal N$
for every $\gamma \in \mathcal X$, where $c_\gamma$ is a
strictly positive
 constant depending continuously on $\gamma \in \mathcal X$.
\end{enumerate}
\end{lemma}

\begin{proof}
(1)$\Rightarrow$(2)
Define a measure $\mathcal N$ on $\mathcal X$ 
by setting $\langle\mathcal N, g\rangle := c_g$
for every continuous function $g\colon \mathcal X\to \mathbb R$, where $c_g$ is as in (1).
We can extend $\mathcal N$ to a measure on $\mathcal J$ by setting $\mathcal N (\mathcal J \setminus \mathcal X)=0$. Such measure is positive since $c_g \geq 0$ for every non-negative $g$.

\medskip

For every $\gamma \in \mathcal X$ and continuous function $g\colon \mathcal X\to \mathbb R$,
we have
\[
\langle
(\Lambda_\psi^*)^n\delta_\gamma, g\rangle=
\langle \delta_\gamma, \Lambda^n_\psi (g)\rangle
\to
\langle \delta_\gamma, c_g \theta\rangle= \theta(\gamma)c_g.
\]
By the definition of $\mathcal N$, and setting $c_\gamma:= \theta (\gamma)$, 
this shows that $(\Lambda_\psi^*)^n \delta_\gamma \to c_\gamma \mathcal N$.

\medskip

(2)$\Rightarrow$(1)
Define a function $\theta\colon \mathcal X\to \mathbb R$ by $\theta(\gamma) := c_\gamma$, for every $\gamma \in \mathcal X$, where $c_\gamma$
 is as in (2).
For every $\gamma \in \mathcal X$ and continuous function $g\colon \mathcal J \to \mathbb R$ we have
\[
\Lambda_\psi^n (g) (\gamma)
=
\langle \delta_\gamma, \Lambda_\psi^n (g)\rangle
=
\langle 
(\Lambda_\psi^*)^n \delta_\gamma, g\rangle
\to \langle c_\gamma  \mathcal  N, g\rangle=
c_\gamma\langle \mathcal N, g\rangle.
\]
 By the definition of $\theta$, and setting $c_g:= \langle\mathcal N,g\rangle$, this shows that
 $\Lambda_\psi^n (g)\to c_g \theta$.
\end{proof}

\begin{remark} The two assertions in Lemma \ref{l:rho-N-equivalent} stay equivalent if $\theta$
is assumed to just be non-negative in (1) and $c_\gamma$ is assumed to just be non-negative in (2).
\end{remark}

For every graph $\gamma \in \mathcal X$ and every integer $n\in \N$, define the measure

\begin{equation}\label{e:web_definition}
\mathcal M_{\gamma,n}
:=
\theta \cdot \theta(\gamma)^{-1}
\cdot (\Lambda_\psi^*)^n \delta_\gamma =
\theta(\gamma)^{-1} \sum_{\mathcal F^n (\gamma')}
e^{\psi(\gamma') + \ldots + \psi(\mathcal F^{n-1} (\gamma'))} \theta(\gamma')\delta_{\gamma'}.
\end{equation}

\begin{lemma}\label{l:criterion-conditions}
If the  conditions in Lemma \ref{l:rho-N-equivalent} are satisfied
and $\theta$ and $\mathcal N$ are as in that lemma, 
the measure $\mathcal M := \theta \mathcal N$ is well defined and the following properties hold:
\begin{enumerate}
\item $\Lambda_\psi\theta = \theta$;
\item $\Lambda_\psi^* \mathcal N= \mathcal N$;
\item $\mathcal M$ is a probability measure and $\mathcal M (\mathcal J_s)=0$;
\item $\mathcal M$ is $\mathcal F$-invariant;

\item 
$\mathcal M_{\gamma,n} \to \mathcal M$
for every
$\gamma \in \mathcal X$;
\item $\F^{-1}(\Supp \M) \subseteq \Supp \M$.
\end{enumerate}
\end{lemma}

\begin{proof}
(1)
This is a consequence of the first condition in Lemma \ref{l:rho-N-equivalent}. Indeed, as $\Lambda_\psi$ is continuous on $\mathcal C^0 (\mathcal X)$, we have
\[
c_\theta \theta =
\lim_{n\to \infty} \Lambda_\psi^{n+1} \theta
=
\Lambda_\psi 
(\lim_{n\to \infty} \Lambda_\psi^n \theta) =
\Lambda_\psi (c_\theta \theta)
= c_\theta \Lambda_\psi  \theta.
\]
As $c_\theta>0$, we deduce 
that $\theta = \Lambda_\psi (\theta)$, as desired.
Observe in particular that this implies that $c_\theta=1$, 
since  $\Lambda_\psi^n \theta=\theta$ for every $n\in \mathbb N$  and $\Lambda_\psi^n \theta \to c_\theta \theta$. 

\medskip

(2)
Observe that $\Lambda_\psi^* \mathcal N$ is well defined since $\mathcal N(\mathcal J_s)=0$, and still
satisfies  $(\Lambda_\psi^* \mathcal N) (\mathcal J_s)=0$. It is enough to check that 
$\langle \Lambda_\psi^* \mathcal N, g\rangle = \langle\mathcal N, g\rangle$ for every $g \in \mathcal C^0 (\mathcal J)$.
With the notation of Lemma \ref{l:rho-N-equivalent}, we also have $\langle\mathcal N, g\rangle=c_g$, where $c_g$ is characterized by the convergence $\Lambda_\psi^n g \to c_g \rho$.
Observe in particular that $c_g = c_{\Lambda_\psi g}$. It follows that,
for every $g$ as above, we have
\[
\langle
 \Lambda_\psi^* \mathcal N, g
\rangle =
\langle
\mathcal N, \Lambda_\psi g\rangle
=
c_{\Lambda_\psi g}= c_g = \langle\mathcal N, g\rangle.
\]

(3)
As $\theta$ is positive on $\mathcal X$, $\mathcal N(\mathcal J_s)=0$, and   $ \langle\mathcal N, \theta\rangle=c_\theta =1$, we have  $\theta \in L^1 (\mathcal N)$.
As $\mathcal N(\mathcal J_s)=0$, we have $\mathcal M (\mathcal J_s)=0$. As $\theta$ is non-negative and $\mathcal N$ is a positive measure, $\mathcal M$ is a positive measure. It is a probability measure since
\[\mathcal M (\mathcal X) = \langle\theta\mathcal N, \mathbb 1_{\mathcal X}\rangle
=
\langle\mathcal N, \theta\rangle = c_\theta =1.\]

(4)
We need to check that $\langle\mathcal M, g\rangle
=
\langle \M,g \circ \mathcal F \rangle$
for every continuous function $g$ on $\mathcal J$.
As $\mathcal M (\mathcal J_s)=0$, the pairing can be computed on $\mathcal X$. A direct computation gives that
\[
\langle \mathcal M, g \circ \mathcal F
\rangle
=
\langle \theta \mathcal N, g \circ \mathcal F\rangle
=
\langle \mathcal N, \theta\cdot g \circ \mathcal F\rangle
=
\langle \mathcal N, \Lambda_\psi( \theta\cdot g \circ \mathcal F)\rangle,\]
where in the last step we used the equality
$\Lambda_\psi^*\mathcal N = \mathcal N$. It follows from the definition of $\Lambda_\psi$ and the $\Lambda_\psi$-invariance of $\theta$
that
\[
\Lambda_\psi (\theta \cdot g \circ F) = g \cdot \Lambda_\psi (\theta) = g \theta.
\]
Together with the previous equalities, this gives
\[
\langle \mathcal M, g \circ \mathcal F
\rangle
=
\langle
\mathcal N, \theta  g\rangle =
\langle\mathcal M, g\rangle,
\]
as desired.

\medskip

(5)
Take $g\in \mathcal C^0 (\mathcal J)$. By the second assertion of Lemma \ref{l:rho-N-equivalent}, and recalling that $c_\gamma=\theta(\gamma)$ for every
$\gamma\in \mathcal X$,
we have
\[
\langle \mathcal M_{\gamma, n}, g\rangle
=
\theta(\gamma)^{-1}
\langle
\theta 
\cdot (\Lambda_\psi^*)^n \delta_\gamma, g
\rangle
\to
\theta(\gamma)^{-1}
\langle c_\gamma \mathcal N, \theta g\rangle
=
\langle\mathcal M, g\rangle.
\]
The assertion follows.

\medskip

(6)
Since $\theta$ is positive, we have $\Supp \M = \Supp \mathcal N$. Take $\gamma \in \Supp \mathcal N$ and $\gamma' \in \J$ such that $\F(\gamma') = \gamma$.
Fix a small open set $U$ containing $\gamma$ and let $U'$ be the connected component of $\mathcal F^{-1} (U)$ containing $\gamma'$ (which is open since $\mathcal F$ is continuous).
Since $\gamma \in \Supp \mathcal N$, we have $\mathcal N(U)>0$.
By the definition of $\Lambda_\psi$, we have that
$\Lambda_\psi^*\mathcal N (U')>0$.
This gives that $\gamma' \in \Supp \Lambda^*_\psi \mathcal N$.
By (2), it follows that $\gamma'\in \Supp \mathcal N$, as desired.

\end{proof}




\subsection{Estimates of contraction along inverse branches.}
We assume in the following that the family $(M,f)$
admits
a web $\mathcal M$, i.e., an
$\mathcal F$-invariant compactly supported
probability measure on $\mathcal J$, satisfying the following properties:

\begin{itemize}
\item[{\bf (M1)}] $\mathcal M$ is \emph{acritical}, i.e., we have $\mathcal M (\mathcal J_s)=0$;
\item[{\bf (M2)}] there exists a constant
$A_1>0$ such that, 
for every $\lam \in M$, 
the probability measure $(p_\lam)_* (\mathcal M)$ is ergodic and  
the Lyapunov exponents
of $(p_\lam)_* (\mathcal M)$ are strictly larger than $A_1$;
\item[{\bf (M3)}]
there
exists a continuous
function $\psi\colon \mathcal X \to \R$, a strictly positive continuous function 
$\theta\colon  \mathcal X\to \mathbb R$, and a positive measure $\mathcal N$ on $\mathcal J$ 
with $\mathcal N (\mathcal J_s)=0$
such that
\begin{enumerate}
\item 
$\mathcal M = \theta \mathcal N$,
\item $\Lambda_\psi \theta= \theta$,
\item $\Lambda_\psi^* \mathcal N= \mathcal N$;

\item for $\mathcal M$-almost every $\gamma \in \mathcal J$, we have
$\mathcal M_{\gamma,n} \to \mathcal M$,
where $\M_{\gamma, n}$ is defined as in \eqref{e:web_definition}.
\end{enumerate}
\end{itemize}



\medskip


Observe that condition {\bf (M1)} is sufficient to imply  the stability of the family $(M,f)$ in the sense of Theorem-Definition \ref{t:bbd}, see \cite[Theorem 4.1]{BBD18}.
Conversely, a stable family always admits at least a web $\mathcal M$ satisfying the above assumptions.
To this purpose, it is enough to consider an acritical equilibrium web $\mathcal M_0$, 
as constructed in \cite{BBD18}. This web corresponds to the case of $(p_\lam)_* \mathcal M= \mu_\lam$
(the measure of maximal entropy) for all $\lam \in M$, and $\psi \equiv - k \log d$.
An asymptotic contraction property along generic inverse branches for 
$\mathcal M_0$ is proved in \cite[Proposition 4.2 and 4.3]{BBD18}. This property is the key to getting the measurable holomorphic motion in Theorem-Definition \ref{t:bbd}. It was generalized in \cite{BR22} for the
larger class of webs satisfying {\bf (M1)} and {\bf (M2)}, see Proposition \ref{p:bbd-gen} below.
We aim here at establishing a more quantitative version of such results, under the extra condition {\bf (M3)}. In particular, all this will also apply 
to the equilibrium web $\mathcal M_0$ as above.

\medskip


Given 
$\Omega\subset M$, 
$\gamma\in \mathcal X$ and $\eta>0$,
we denote by
$T_\Omega(\gamma, \eta)$ the $\eta$-neighbourhood of the graph $\Gamma_\gamma$ of $\gamma$ in $\Omega \times \P^k$, i.e.,
\[
T_\Omega (\gamma, \eta) := \{
(\lambda, z) \in \Omega \times \P^k\colon \dist_{\P^k}(z,\gamma(\lambda))< \eta
\}.
\]
We call such neighbourhood a \emph{tube} at $\gamma$ over $\Omega$. Observe that a tube $T_\Omega(\gamma, \eta)$ corresponds to the ball $\mathcal B_\Omega (\gamma, \eta)$ in the metric space $(\mathcal J, \dist_{\Omega})$,
where the distance $\dist_\Omega$ is given by
\begin{equation}\label{e:distance-Omega}
\dist_\Omega (\gamma_1, \gamma_2) := \sup_{\lam \in \Omega}
\dist_{\P^k} (\gamma_1(\lambda),\gamma_2 (\lambda)).
\end{equation}




Given a tube $T= T_\Omega (\gamma,\eta)$,
the \emph{slice}
$T_{|\lambda}$ is the ball
$B(\gamma(\lam), \eta) = T \cap (\{\lambda\}\times \P^k)$.
More generally, given the image of a tube $T$ by a holomorphic map $g\colon T\to \Omega\times \P^k$ fibered over $M$, we define the slice $g(T)_{|\lam}$
of $g(T)$ at $\lambda$
as $g(T)\cap \{\lam\}\times \P^k$.


\medskip

We fix in what follows a constant $0<A_0<A_1$, where $A_1$ is given in {\bf (M2)}.

\begin{definition}
Given $\Omega\subset M$,
$\gamma \in \mathcal X$, a tube $T$ at $\gamma$ over $\Omega$,
and $n \in \mathbb N$, we say that a map
$g \colon T\to g(T)$ is a 
\emph{$m$-good inverse branch of $f$ of order $n$ on $T$} if
\begin{enumerate}
\item $g \circ f^n = id_{g(T)}$;
\item for all $\lam\in \Omega$, $\diam f^l_\lam (g(T)_{|\lam})\leq e^{-m- (n-l)A_0}$ for all $0 \le l \leq n$.
\end{enumerate}
\end{definition}

Observe that, by definition, we have diam $T_{|\lam}\leq e^{-m}$ for every tube $T$ admitting a $m$-good inverse branch as above. 

\medskip

Given an inverse branch $g$
of $f^m$ defined on a tube $T$,
given any $\gamma\in \mathcal J$ with $\Gamma_\gamma \subset  T$ 
we can in particular associate to such 
inverse branch a map
$\gamma_g$
such that $\Gamma_{\gamma_g}\subset g(T)$ and $\mathcal F (\gamma_g)=\gamma$. In particular, the association $\gamma \mapsto \gamma_g$ defines a map $\mathcal G$ on the ball $\mathcal B_\Omega (\gamma,\eta)$, that we can see as an inverse branch for $\mathcal F$ over such ball.
%
%


Given $\Omega\subset M$ and a tube $T$ at $\gamma\in \mathcal X$ over $\Omega$,
we denote by $\mathcal M_{T,n}^{(m)}$
the measure
\begin{equation}\label{e:M-good-branches}
\mathcal M^{(m)}_{T,n}:=
\theta (\gamma)^{-1}
\sum_{\gamma_g}
e^{\psi(\gamma_g)+ \ldots +\psi(\mathcal F^{n-1}(\gamma_g))}
\theta (\gamma_g) \delta_{\gamma_g},
\end{equation}
where the sum is over the preimages $\gamma_g$  of $\gamma$
associated to $m$-good inverse branches $g$ of $f$
of order $n$ on $T$.

\begin{remark}
We have $\mathcal M^{(m)}_{T,n}\leq
\mathcal M_{\gamma,n}$ for all $n\geq 0$. Hence, the condition \textbf{(M3)} ensures that any limit value $\mathcal M'_T$ of the sequence 
$\{\mathcal M^{(m)}_{T,n}\}_{n}$
satisfies $\|\mathcal M'_T\|\leq 1$.
\end{remark}

\begin{definition}
Given $\Omega \subset M$ and $m>1$, we say that a tube $T$ at $\gamma\in \mathcal X$ over $\Omega$
is 
\emph{$m$-nice} 
if
$\|\mathcal M^{(m)}_{T,n}\|\geq 1-1/m$ for all $n$ sufficiently large.
We say that a ball $\mathcal B_\Omega(\gamma, \eta)$ is $m$-nice if the tube $T_\Omega(\gamma, \eta)$ is $m$-nice.
\end{definition}



\begin{remark}
Every $m$-nice tube $T(\gamma,\eta)$ satisfies $\eta\leq e^{-m}$, and $\|\mathcal M'_{T}\|\geq 1-1/m$ for every limit value $\mathcal M'_T$ of the sequence $\mathcal M^{(m)}_{T,n}$.
\end{remark}

The following is the main result of this section, giving a quantitative control on the contraction along $\mathcal M$-generic
inverse branches of $\mathcal F$.

\begin{proposition}\label{p:l411}
Let $\mathcal M$ be a compactly supported and 
ergodic 
web
satisfing {\bf (M1)}, {\bf (M2)}, and {\bf (M3)}.
For every $\Omega \subset M$ and $\gamma\in \mathcal X$, the tube $T_\Omega(\gamma,\eta)$
is $m$-nice if $\eta$ is sufficiently small.
\end{proposition}

In order to prove Proposition \ref{p:l411}, we will make use of the natural extension $(\hat \J, \hat \F, \hat \M)$ of the system $(\mathcal J, \mathcal F, \mathcal M)$.
We refer to \cite[Section 10.4]{CFS} and \cite[Theorem 2.7.1]{PU} for the general setting of Lebesgue spaces.
The main difficulty lies in proving that $\hat \M$ is $\sigma$-additive.
 We recall here how to construct the extension of the Polish space $(\J,\dist_\J)$, introducing some notations that will be used later.

\medskip


We denote by $\hat {\mathcal J}$ the subspace
\[
\hat {\mathcal J} := \{\hat \gamma = (\gamma_n)_{n\in \mathbb Z} \in \mathcal J^\Z \colon \mathcal F (\gamma_n) = \gamma_{n+1}, n\in \Z\},\]
by $\pi_n\colon \hat{\mathcal J} \to \mathcal J$ the projection defined as $\pi_n (\hat \gamma)=\gamma_n$ and by $\hat {\mathcal F}\colon\hat {\mathcal J}\to \hat{\mathcal J}$ the shift map
\[
\hat{\mathcal F}
(\hat \gamma)
=
(\mathcal F (\gamma_n))_{n\in \mathbb Z} = (\gamma_{n+1})_{n\in \mathbb Z}.
\]
The maps satisfy $\pi_n \circ\hat {\mathcal F}= \mathcal F \circ \pi_n$ for all $n\in \mathbb Z$.
We endow $\J$ with its Borel $\sigma$-algebra.
If $\mathcal B \subset \J$ is a Borel set, define $\mathcal A_{n,\mathcal B} := \pi_n^{-1} (\mathcal B)=\{
\hat \gamma \in \hat \J \colon \gamma_n \in \mathcal B
\} \subset \hat {\J}$.
We may consider the set $\mathcal A_{n,\mathcal B}$ as a subset of $\hat \J^{\{-n,\dots,n\}} := \{(\gamma_k)_{-n\le k\le n} \in \J^{\{-n,\dots,n\}}\: : \: \F(\gamma_k) = \gamma_{k+1}, -n\le k < n \}$ as well.

A \emph{cylinder} is a finite intersection of subsets of  $\hat {\mathcal J}$ of the form $\mathcal A_{n,\mathcal B}$,
for some $n\in \Z$ and open set $\mathcal B \subseteq \mathcal J$.
The family of all the cylinders forms a basis for the product topology.
It is also a $\pi$-system
(see \cite[Chapter 1]{Kal}) whose generated $\sigma$-algebra is the Borel $\sigma$-algebra
of $\hat{\mathcal J}$.
By the monotone classes theorem  (see for instance \cite[Theorem 1.1]{Kal}), any measure on $\hat \J^{\{-n,\dots, n\}}$ is thus defined by its values on the cylinders.
Given $\mathcal B_{-n}, \dots, \mathcal B_n \subset \J$ open sets, define
\begin{equation}\label{e:def-hat-M-cylinder}
\hat {\mathcal M}_n (\mathcal A_{-n,\mathcal B_{-n}} \cap \dots \cap \mathcal A_{n, \mathcal B_n})
:=
\mathcal M (\mathcal B_{-n} \cap \F^{-1}(\mathcal B_{-(n-1)}) \cap \dots \cap \F^{-2n}(\mathcal B_{n})).
\end{equation}

The probability measure $\hat {\mathcal M}_n$ is well defined on $\hat \J^{\{-n,\dots,n\}}$.
The invariance of ${\mathcal M}$ and the fact that $\gamma_k \in \mathcal B$ if and only if $\gamma_{k-1} \in \F^{-1}(\mathcal B)$
yield the following consistency condition
\[
\hat \M_n (\mathcal A_{k,\mathcal B}) = \hat \M_{n+1} (\mathcal A_{k,\mathcal B}),
\]
for every Borel set $\mathcal B \subset \J$ and integers $n\in \N$, $k \in \{-n,\dots, n \}$.
By Kolmogorov extension theorem,
the measure $\hat \M$ then admits a unique extension on $\hat \J$.

Observe that $\hat {\mathcal M}$ is then
$\hat{\mathcal F}$-invariant and satisfies $(\pi_n)_* \hat{\mathcal M} = \mathcal M$ for every $n\in \Z$.





\medskip

In the following, we will use the disintegration of $\hat {\mathcal M}$ with respect to $\mathcal M$ and the projection $\pi_0$, and more precisely 
 the \emph{conditional measures} $\hat{\mathcal M}^\gamma$ of $\hat {\mathcal M}$ on $\{\gamma_0=\gamma\}$. These measures
are uniquely defined for $\mathcal M$-almost every $\gamma \in \mathcal J$ and are characterized by the property that
\[
\langle
\hat {\mathcal M} , g\rangle
= \int_{\mathcal X}
\langle \hat{\mathcal M}^\gamma ,g\rangle
\mathcal M(\gamma) 
\]
for all non-negative bounded measurable functions $g \colon \mathcal J \to \mathbb R$
(see \cite{B23} for more details about the construction).
We now describe a useful approximation of the conditional measures $\hat{\mathcal M}^\gamma$, that we will need later.

\medskip

For every $n>0$, consider the projection $\pi^n\colon \hat {\mathcal J} \to \mathcal J^{n+1}$ defined as
\[
\pi^n := (\pi_{-n},\dots, \pi_{0} ).
\]
By {\bf (M1)}, 
$\mathcal X$ satisfies $\mathcal M (\mathcal X)=1$ and the map
$\mathcal F\colon \mathcal X \to \mathcal X$ is well defined and surjective.
For every $(\gamma_{-n}, \dots, \gamma_0)\in \mathcal X^{n+1}$ with $\mathcal F (\gamma_j) = \gamma_{j+1}$ for every $-n \leq j\leq -1$, we choose a representative $\hat \beta \in \hat {\mathcal J}$ such that $\beta_j = \gamma_j$ for all $-n\leq j\leq 0$. Observe that, given any $\gamma_0 \in$ $\mathcal X$, there are $d^{kn}$ elements as above in $\mathcal X^{n+1}$, hence $d^{kn}$ representatives. We denote by
$\hat {\mathcal Z}_n$ the collection of such representatives.


\medskip

For every $\gamma \in \mathcal X$, define 
\[
\hat{\mathcal M}^\gamma_n
:=
\theta(\gamma)^{-1}
\sum_{ \hat \beta \in \hat{\mathcal Z}_n \colon \beta_0 = \gamma }
e^{\psi (\beta_{-n}) + \ldots + \psi(\beta_{-1})}
\theta(\beta_{-n})
\delta_{\hat z}.
\]
Observe in particular that, by {\bf (M1)} and the fact that $\theta$ is positive on $\mathcal X$,
$\hat { \mathcal M}^\gamma_n$ is defined for $\mathcal M$-almost every $\gamma \in \mathcal J$.

\begin{lemma}\label{l:disintegration}
For every $\gamma \in \mathcal X$, we have
\[
\lim_{n\to \infty} \hat{\mathcal M}^\gamma_n  =\hat {\mathcal M}^{\gamma}.
\]
\end{lemma}

\begin{proof}
It is enough to show that, for $\gamma \in \mathcal X$, we have
\[
\lim_{n\to \infty} \hat{\mathcal M}^\gamma_{n} (\mathcal A_{-i, \mathcal B}) = \hat{\mathcal M}^\gamma (\mathcal A_{-i,\mathcal B})
\mbox{ for all }  i \geq 0 \mbox{ and Borel set } \mathcal B.
\]
Hence, the assertion is
a consequence of the following two identities:

\begin{enumerate}
\item $\hat {\mathcal M}^\gamma_n (\mathcal A_{-i,\mathcal B})=
\hat {\mathcal M}^\gamma_i (\mathcal A_{-i,\mathcal B})$
for every $\gamma \in \mathcal X$ and all $n> i \geq 0$;
\item $\int \hat{\mathcal M}^\gamma_i (\mathcal A_{-i, \mathcal B}) \mathcal M (\gamma)= \mathcal M (\mathcal B) $ for all $i\geq 0$.
\end{enumerate}

\medskip

We start proving the first identity above. This is a consequence of the identity $\Lambda_\psi \theta = \theta$. Indeed, for every $n>i\geq 0$ and Borel set $\mathcal B$, we have
\[
\begin{aligned}
\hat{\mathcal M}^\gamma_n (\mathcal A_{-i, \mathcal B})
&= \theta(\gamma)^{-1}
\sum_{\hat \beta \in \hat {\mathcal Z}_n \colon \beta_0=\gamma}
e^{\psi (\beta_{-n}) + \ldots + \psi (\beta_{-1})}
\theta (\beta_{-n})
 \delta_{\hat \beta} (\mathcal A_{-i, \mathcal B})\\
 &= \theta(\gamma)^{-1}
 \sum_{\hat \beta \in \hat {\mathcal Z}_i \colon \beta_0=\gamma}
 (\Lambda_{\psi}^{n-i}\theta) (\beta_{-i})
e^{\psi (\beta_{-i}) + \ldots + \psi (\beta_{-1})}
 \delta_{\hat \beta} (\mathcal A_{-i, \mathcal B})\\
 &=\theta(\gamma)^{-1}
 \sum_{\hat \beta \in \hat {\mathcal Z}_i \colon \beta_0=\gamma}
 \theta (\beta_{-i})
e^{\psi (\beta_{-i}) + \ldots + \psi (\beta_{-1})}
 \delta_{\hat \beta} (\mathcal A_{-i, \mathcal B})\\
 &=\hat{\mathcal M}^\gamma_i (\mathcal A_{-i, \mathcal B}),
\end{aligned}
\]
which proves the first identity.

\medskip

Let us now prove the second identity. This time, we will use
the fact that $\mathcal N$ is a fixed point for $\Lambda^*_\psi$.
For all $i\geq 0$,
we have
\[
\begin{aligned}
\int \hat{\mathcal M}^\gamma_i (\mathcal A_{-i, \mathcal B}) \mathcal M (\gamma)
&=
\int 
\Big(
\theta(\gamma)^{-1}
\sum_{\hat \beta \in \hat {\mathcal Z}_i \colon \beta_0=\gamma}
 \theta (\beta_{-i})
e^{\psi (\beta_{-i}) + \ldots + \psi (\beta_{-1})}
 \delta_{\hat \beta} (\mathcal A_{-i, \mathcal B})
\Big) \mathcal M(\gamma)\\
&=
\int 
\Big(
\theta(\gamma)^{-1}
\sum_{\mathcal F^i (\gamma')=\gamma}
 \theta ({\gamma'})
e^{\psi (\gamma') + \ldots + \psi (\mathcal F^{i-1} (\gamma'))}
\mathbb 1_{\mathcal B} (\gamma)
\Big) \mathcal M(\gamma)
\\
&=
\big\langle \mathcal M, \theta^{-1}
\Lambda^i_\psi (\theta \mathbb 1_{\mathcal B})
\big\rangle
= 
\big\langle
\mathcal N,\Lambda_\psi^i(\theta \mathbb 1_{\mathcal B}) 
\big\rangle
\\
& = \langle  (\Lambda_\psi^i)^* \mathcal N, \theta \mathbb 1_{\mathcal B}\rangle
= \langle \mathcal N, \theta \mathbb 1_{\mathcal B}\rangle = \mathcal M (\mathcal B),
\end{aligned}
\]
where we denoted by $\mathbb 1_{\mathcal B}$ the indicatrix function of $\mathcal B$ and in the last step we used the identity $\mathcal M = \theta \mathcal N$. The proof is complete.
%
%
%
%
%
%
%
\end{proof}

For every $n\geq 0$, we denote by $f_{\hat \gamma}^{-n}$ the inverse branch of $f^n$, defined in a neighbourhood of $\Gamma_\gamma$, and such that $f_{\hat \gamma}^{-n} (\gamma(\lam)) = \gamma_{-n} (\lam)$ for all $\lam \in M$.
Such branch exists for every $\gamma \in \mathcal X$, but a priori it is only defined in some (non-controlled) 
neighbourhood of $\Gamma_{\gamma}$.
The following proposition (see
 \cite[Propositions 4.2 and 4.3]{BBD18} for the case of the equilibrium web and 
 \cite[Proposition A.1]{BR22}
 for the general case) 
gives a uniform control in $\lambda$
on the size of the neighbourhood of $\gamma (\lambda)$
where such inverses are defined.






\begin{proposition}\label{p:bbd-gen}
Let $(M,f)$ be a stable
holomorphic family of endomorphisms of $\mathbb P^k$.
Let $\mathcal M$ be 
a web satisfying {\bf (M1)} and  {\bf (M2)}.
%
%
 Then, for every open set $\Omega \Subset M$ and $0<A<A_1$,
 there exists
 a Borel subset $\hat {\mathcal Y}\subseteq \hat {\mathcal J}$ with
 $\hat {\mathcal M} (\hat { \mathcal Y})=1$, and two
 measurable functions $\hat \eta_{A} \colon \hat {\mathcal Y} \to ]0,1]$
 and $\hat l_A \colon \hat {\mathcal Y}\to [1,+\infty[$
 which satisfy the following properties.

 For every $\hat \gamma \in \hat{\mathcal Y}$ and every $n\in  \mathbb N^*$ the iterated inverse branch $f_{\hat \gamma}^{-n}$ is defined and Lipschitz on the tubular neighbourhood 
 $T_{\Omega} (\gamma_0, \hat \eta_A (\hat \gamma))$
 of the graph $\Gamma_{\gamma_0}\cap (\Omega \times \P^k)$ of $\gamma_0$, 
 and we have
 \[
f_{\hat \gamma}^{-n} (T_{\Omega} (\gamma_0, \hat \eta_A (\hat \gamma) ))\subset T_{\Omega} (\gamma_{-n}, e^{-nA})
 \quad
 \mbox{ and } \quad
 \widetilde \Lip ( f^{-n}_{\hat \gamma} )\leq \hat l_A  (\hat \gamma) e^{-nA},
 \]
 where $\widetilde \Lip
 ( f^{-n}_{\hat \gamma} ):= \sup_{\lam \in \Omega} \Lip \big( (f_{\hat \gamma}^{-n})_{|B(\gamma_0(\lam), \hat \eta_A)} \big)$.
\end{proposition}

We can now prove Proposition \ref{p:l411}.

\begin{proof}[Proof of Proposition \ref{p:l411}]
%
%
Fix $m>0$, a constant $A_0< A <A_1$ 
and a positive integer $r$.
For every $N\in \mathbb N$, define
\[
\hat {\mathcal Y}_N := \{
\hat \gamma \in \hat{\mathcal Y} \:\colon\:
\eta_A (\hat \gamma) \geq N^{-1} \mbox{ and } l_A (\hat \gamma)\leq N 
\}.
\]
By Proposition \ref{p:bbd-gen},
we 
have $\lim_{N\to \infty} \hat{\mathcal M} (\hat{\mathcal Y}_N)=1$. 
Fix $N_0 = N_0 (m, r)$ such that
$\hat{\mathcal M}(\hat{\mathcal Y}_N)> 1-1/(2m^{r+1})$
for all $N\geq N_0$. Then, by Markov inequality and the definition of the measures $\hat {\mathcal M}^\gamma$,
we see that there exists a subset $\mathcal X_r \subset\mathcal X$
such that $\mathcal M (\mathcal X_r) > 1-1/m^r$ and
\begin{equation}\label{e:disintegration-error-m}
\hat {\mathcal M}^\gamma
(\hat {\mathcal Y}_N\cap \{\gamma_0 = \gamma\})
> 1-1/(2m) 
\quad
\mbox{ for all } \gamma \in \mathcal X_r
\quad \mbox{ and } \quad
N\geq N_0.\end{equation}
In order to prove the statement, it is enough to prove the assertion for all $\gamma \in\mathcal X_r$.
Fix one such $\gamma$. It follows from
Proposition \ref{p:bbd-gen}
that all inverse branches defined 
on $T_\Omega (\gamma, e^{-m}/(2N))$ and
corresponding to $\hat \gamma \in \hat {\mathcal Y_N}\cap \{\gamma_0=\gamma\}$ are $m$-good for all $n$.
It follows from Lemma \ref{l:disintegration} and \eqref{e:disintegration-error-m} that
\[
\hat {\mathcal M}^\gamma_n
(\hat {\mathcal Y}_N\cap \{\gamma_0 = \gamma\})
> 1-1/m 
\quad
\mbox{ for all } 
n  \mbox{ large enough}.
\]
This implies that, for all $n$ sufficiently large, we have
$\|\mathcal M^{(m)}_{T,n} \|> 1-1/m$, 
for $T= T_\Omega(\gamma, e^{-m}/(2N))$. By definition, this means that $T$ is $m$-nice.
\end{proof}

\section{Proof of Theorem \ref{t:main}}









\subsection{Equilibrium states and associated webs}
\label{ss:equilibrium}
We work here in the assumptions of Theorem \ref{t:bbd}.
In particular, we assume that $f_{\lam_0}$
satisfies condition {\bf (A)} and $\phi$ satisfies the
conditions in {\bf (B)} in the Introduction.
Up to replacing $\phi$ by $\phi- P(\phi)$, we can assume that $P(\phi)=0$.
We denote by $\Lambda_\phi$ the (Ruelle-Perron-Frobenius) transfer operator
\[
\Lambda_{\phi} (g) (y) := \sum_{f_{\lam_0}(x)=y}e^{\phi (x)} g(x)
\]
acting on $\mathcal C^0 (\P^k)$, where the elements in the sum are counted with multiplicity.
By \cite[Theorem 1.1]{BD23}, there exists 
a unique 
probability measure $m_\phi$
and
a unique (up to multiplicative constant) strictly positive
continuous function $\rho_\phi \colon \P^k \to \R$ such that
\[
\Lambda_\phi^n (g) \to \langle m_\phi, g \rangle \rho_\phi
\quad \mbox{ and }
\quad (\Lambda_\phi^n)^* \delta_x \to \rho_\phi (x) m_\phi
\]
for all continuous function $g\colon \P^k \to \R$ and all $x \in \P^k$.
We normalize $\rho_\phi$ so that $\langle m_\phi, \rho_\phi\rangle=1$. The probability measure $\mu_\phi := \rho_\phi m_\phi$ is then $f_{\lam_0}$-invariant, and is the 
unique 
equilibrium state  for $f_{\lam_0}$
associated to $\phi$.
In particular, for every
$x \in \P^k$, the measures  
\[
\mu_{x,n} := \rho_\phi \cdot 
\rho_\phi(x)^{-1}  
(\Lambda_\phi^n)^*  \delta_x=
\rho_\phi^{-1}(x)\sum_{f^n(y)=x} e^{\phi(y)+ \ldots +\phi(f^{n-1} (y))} \rho_\phi(y)\delta_y
\]
satisfy
$\mu_{x,n}\to \mu_\phi \mbox{ as } n \to \infty$.

\medskip

By \cite[Proposition 4.9]{BD23}, $\mu_\phi$ satisfies
$h_{\mu_\phi}> \log d^{k-1}$. Hence, by
\cite{BR22}, there exists an ergodic 
web $\mathcal M_{\lam_0,\mu_\phi}$ on $\mathcal J$ with the property that
$(p_{\lam_0})_{*} \mathcal M_{\lam_0, \mu_\phi}=\mu_\phi$
and
which satisfies {\bf (M1)} and {\bf (M2)}.
The following lemma, proved in \cite{BR}, gives in particular the uniqueness of such web.

\begin{lemma}\label{l:strong-uniqueness}
Let $(M,f)$ 
be a stable family of endomorphisms of $\mathbb P^k$. Let $\mathcal M_1$ and $\mathcal M_2$ be two
probability measures compactly supported on $\mathcal J$. Assume that $\mathcal M_1$ is an acritical web and that there exists $\lam_0 \in M$ such that
$(p_{\lam_0})_* \mathcal M_1= (p_{\lam_0})_*\mathcal M_2$.
Then $\mathcal M_1 = \mathcal M_2$.
%
\end{lemma}

Define the functions
$\theta, \psi \colon \mathcal J \to \R$ as
\[
\theta (\gamma) := \rho_\phi (\gamma (\lam_0))
\quad
\mbox{ and }
\quad 
\psi (\gamma) :=
\phi(\gamma (\lam_0)).
\]
Observe that these functions are continuous on $\mathcal J$.
In particular, for every $m\in \N^*$, any graph $\gamma \in \J$ admits an open neighbourhood $\mathcal A \subset \J$ such that
\begin{equation}\label{e:small_enough}
\inf_{\mathcal A}\theta \geq (1-1/m) \sup_{\mathcal A}\theta.
\end{equation}
To see this, take an open neighbourhood $\mathcal A' \subset \J$ of $\gamma$.
By the continuity of $\theta$ at $\gamma$, and the fact that $\rho_\phi$ is positive on the compact set $\P^k$, there exists an open neighbourhood $\mathcal A \subset \mathcal A'$ such that
\[
\theta(\gamma_1) - \theta(\gamma_2) \le \frac{\inf_{\mathcal A'} \theta}{m}
\quad 
\mbox{ for any } \gamma_1, \gamma_2 \in \mathcal A.
\]
Observe that $\inf_{\mathcal A'} \theta \le \inf_{\mathcal A}\theta \le \sup_{\mathcal A} \theta$.
Taking respectively the supremum over $\gamma_1$ and the infimum over $\gamma_2$
proves the claim.

\medskip

Consider the operator $\Lambda_\psi$
defined as in \eqref{e:def-Lambda-psi}.
For $\gamma \in \mathcal X$, set $\M_{\gamma,n}$ as defined in \eqref{e:web_definition}.
It follows from the above and Lemmas \ref{l:rho-N-equivalent},
\ref{l:criterion-conditions}, and
\ref{l:strong-uniqueness},
that $\theta$ satisfies $\Lambda_\psi (\theta)= \theta$
and that there exists a unique probability measure $\mathcal N$ on $\mathcal X$
which is a fixed point
for $\Lambda_\psi^*$.
Furthermore, the web $\mathcal M_{\lam_0, \mu_\phi}$ satisfies
$\mathcal M_{\lam_0, \mu_\phi}= \theta \mathcal N$, and for any $\gamma \in \mathcal X$ we have
\begin{equation}\label{e:conv-42-web}
\begin{aligned}
\M_{\gamma,n}
&=
\theta(\gamma)^{-1} 
\sum_{\mathcal F^n (\gamma')=\gamma}
e^{\psi(\gamma') + \ldots + \psi(\mathcal F^{n-1}(\gamma'))}
\theta (\gamma')
\delta_{\gamma'}
\\
&=
\theta (\gamma(\lambda_0))^{-1}
\sum_{\mathcal F^n (\gamma')=\gamma} e^{\phi(\gamma' (\lambda_0)) + \ldots +\phi (\mathcal F^{n-1} (\gamma') (\lambda_0))} \theta (\gamma' (\lambda_0)) \delta_{\gamma'}\\
& \to \mathcal M_{\lam_0, \mu_\phi}
\mbox{ as } n\to \infty.
\end{aligned}
\end{equation}
In particular, $\mathcal M_{\lam_0, \mu_\phi}$ satisfies condition {\bf (M3)}.




\medskip



We conclude this section with the following lemma,
that we will need in the next section. 
Recall that $q>2$
and $\|\phi\|_{\log^q}<\infty$.

\begin{lemma}\label{l:l413}
There exists a positive constant $C=C(A_0,q)$ such that, for all $n\in \mathbb N$, $m\geq 0$, and every $m$-good inverse branch $g\colon T \to g(T)$ of $f$ of order $n$ on a tube $T$, and for all sequences of graphs
$\{\gamma^1_l\}_{0 \leq l\leq  n-1}$ and
$\{\gamma^2_l\}_{0 \leq l \leq n-1}$
with $\Gamma_{\gamma_l^1}, \Gamma_{\gamma_l^2}\subset 
f^l (g(T))$ for all $0\leq l \leq n-1$,
we have
\[
\sum_{l=0}^{n-1}
|\psi (\gamma_l^1) -\psi( \gamma^2_l)|\leq C m^{-(q-1)}.
\]
\end{lemma}

\begin{proof}
It follows 
from the definition of $\psi$ that
\[
\sum_{l=0}^{n-1}
|\psi (\gamma_l^1) -\psi( \gamma^2_l)|
=
\sum_{l=0}^{n-1}
|\phi (\gamma_l^1 (\lam_0)) -\phi( \gamma^2_l (\lambda_0))|.
\]
The assertion is a consequence of \cite[Lemma 4.13]{BD23}.
\end{proof}


\subsection{Construction of repelling graphs}\label{ss:construction}


For simplicity, we denote in this section by $\mathcal M$ the web $\mathcal M_{\lam_0,\mu_\phi}$ defined in the previous section.
We also recall that we are assuming,
with no loss of generality, that $P(\phi)=0$.
We will 
fix a relatively compact open subset $M'\Subset M$, and only consider the graphs of elements of $\mathcal J$ as graphs over $M'$. In particular, all distances and balls are with respect to the distance $\dist_{M'}$, see \eqref{e:distance-Omega}.



The following is the main result of this section and the key estimate to prove Theorem \ref{t:main}. Thanks to the results proved in Section \ref{s:prelim-inverses}, we can follow
the strategy of \cite[Lemma 4.14]{BD22}. In particular we also employ a trick due to Buff \cite{B05} which simplifies the original proof of the equidistribution of repelling points with respect to the measure of maximal entropy in 
\cite{BD99}.

\begin{proposition}\label{p:l414}
Let $\mathfrak U$ be a finite collection of disjoint 
open subsets of $\mathcal J$. For every $m>0$ there exists $n(m, \mathcal \mathfrak U)>m$
and, for every $n> N(m,\mathfrak U)$, a set $\mathcal Q_{m,n}$ of motions of repelling periodic points of period $n$ such that, for all $ \mathcal U\in \mathfrak U$ we have
\[
(1-1/m) \mathcal M (\mathcal U)
\leq
\sum_{\gamma \in \mathcal Q_{m,n}\cap U}
e^{\psi (\gamma) + \ldots + \psi (\F^{n-1} (\gamma))}
\leq 
(1+1/m) \mathcal M (\mathcal U)
\]
\end{proposition}



Before proving Proposition \ref{p:l414}, we make some preliminary simplifications. First of all, it is enough to prove the statement for a single open set $\mathcal U$. The general case follows
setting $n(m,\mathfrak U):= \max_{\mathcal U \in \mathfrak U} n(n,U)$.

Assume that $\mathcal M (\mathcal U)=0$. In that case, it is enough to take
$n(m,\mathcal U)=m+1$ and $\mathcal Q_{m,n}= \emptyset$. Hence, we can assume that $\mathcal M (\mathcal U)>0$.

Fix integers $m_2 \gg m_1 \gg m$. By Proposition \ref{p:l411}, for $\mathcal M$-almost every $\gamma$, the tube $T_{M'}(\gamma, \eta)$ 
(and hence the ball $\mathcal B_{M'}(\gamma, \eta)$ of $\mathcal J$)
is $m_2$-nice if $\eta$ is sufficiently small (depending on $\gamma$).
It follows that, for $\mathcal U$ as above, we can find a finite union of $m_2$-nice balls 
$\mathcal B_i\subset \mathcal J$ with $\mathcal B_i\Subset U$,
satisfying \eqref{e:small_enough},
whose centers belong to $\Supp \M$,
and with the property that
$\mathcal M (\mathcal U \setminus \cup_i \mathcal B_i) < \mathcal M (\mathcal U)/m_2$.
Hence, it is enough to prove the statement for each of the balls  $\mathcal B_i$.
Namely,  the above gives that, in order to prove Proposition \ref{p:l414},
it is enough to 
 to prove the following statement.



\begin{proposition}\label{p:l414-reduced}
For every $m_1>0$ there exists $\bar m_2 (m_1)>m_1$ such that, for all  $m_2> \bar m_2$
%
and 
every
$m_2$-nice ball
$\mathcal T= \mathcal B(\gamma, \eta)$
with $\gamma \in \Supp \M$ and satisfying \eqref{e:small_enough} (with $\mathcal A = \mathcal T$ and $m = m_2$),
there exists $n(m_2)> m_2$ and, for all $n\geq n(m_2)$ a set $\mathcal Q_n\subset \mathcal T\cap \Supp \mathcal M$
of motions of repelling $n$-periodic points 
with the property that
\[
(1-1/m_1) \mathcal M (\mathcal T)
\leq
\sum_{\gamma \in \mathcal Q_n}
e^{\psi (\gamma) + \ldots +\psi (\F^{n-1} (\gamma))}
\leq 
(1+1/m_1) \mathcal M (\mathcal T).
\]
\end{proposition}


Since $\M(\mathcal T) > 0$, we fix an integer $m_3 \ge m_2/ \mathcal M(\mathcal T)$
and a ball $\mathcal T^\star := \mathcal B_{M'}(\gamma, \eta^\star)$, where $\eta^\star < \eta$ is such that 
$\mathcal M( \mathcal T^\star)> (1-1/m_2) \mathcal M (\mathcal T)$.
We also denote by $T$ and $T^\star$
the tubes corresponding to $\mathcal T$ and $\mathcal T^\star$. 
Recalling that the support of $\mathcal M$ is compact,
we also 
choose a finite family of disjoint $m_3$-nice balls 
$\mathcal D_i := \mathcal B_{M'}(\gamma_i, \eta_i)$ such that
$\mathcal M (\cup_i \mathcal D_i) > 1-1/m_3$ and satisfying \eqref{e:small_enough} (with $\mathcal A = \mathcal D_i$ and $m=m_3$).
Set $\mathcal D:= \cup_i \mathcal D_i$.
We also fix balls 
$\mathcal D^\star_i := \mathcal B_{M'}(\gamma_i,\eta^\star_i) \subset \mathcal D_i$ with $\eta_i^* < \eta_i$, $\mathcal M (\cup_i \mathcal D^\star_i)> 1-1/m_3$
and set $\mathcal D^\star := \cup_i \mathcal D^\star_i$.
%
We will denote by $D_i$
the tube corresponding to $\mathcal D_i$ and set $D:=\cup D_i$.

\begin{lemma}\label{l:claim1}
There exists an integer $M_1 = M_1 (m_2, \mathcal T,\mathcal T^\star, \mathcal D_i)$
such that
\[
(1-4/m_2)
\mathcal M (\mathcal T)
\leq
\mathcal M^{(m_3)}_{D_i, N} (\mathcal T^\star)
\leq
(1+4/m_2)
\mathcal M(\mathcal T)
\mbox{ for all } i \mbox{ and } N\geq M_1.
\]
\end{lemma}

Recall that $\mathcal M^{(m_3)}_{D_i,N}$ 
is defined in \eqref{e:M-good-branches}.

\begin{proof}
Since the balls $\mathcal D_i$ are $m_3$-nice
and $m_3 \ge m_2 / \mathcal M (\mathcal T)$,
we have
\[
\mathcal M^{(m_3)}_{D_i,N}
\geq 
(1-\mathcal M(\mathcal T)/m_2)
\quad
\mbox{ for all } i \mbox{ and all } N
\mbox{ large enough}.
\]
As $\mathcal M^{(m_3)}_{D_i,N}\leq \mathcal M_{\gamma_i, N}$ and $\|\mathcal M_{\gamma_i, N}\|= 1 + o(1)$, it follows that
\[
\|\mathcal M_{\gamma_i, N} 
-\mathcal M^{(m_3)}_{D_i, N}\|\leq 
\mathcal M (\mathcal T) / m_2 + o(1).
\]
Hence, 
it is enough to prove that
\[
(1-2/m_2)\mathcal M (\mathcal T)
\leq 
\mathcal M_{\gamma_i, N} (\mathcal T^\star)
\leq 
(1+2/m_2)\mathcal M(\mathcal T)
\mbox{ for all } i \mbox{ and all } N
\mbox{ large enough}.
\]
As $\mathcal M(\mathcal T^\star)\geq (1-1/m_2) \mathcal M(\mathcal T)$, this is a consequence of \eqref{e:conv-42-web}.
\end{proof}

\begin{lemma}\label{l:claim2}
There exists an integer $M_2 = M_2 (m_2, \mathcal T, \mathcal D^\star)$
such that
\[
(1-4/m_2)
\leq
\mathcal M^{(m_2)}_{T, N} (\mathcal D^\star)
\leq
(1+4/m_2)
\mbox{ for all } i \mbox{ and } N\geq M_2.
\]
\end{lemma}

\begin{proof}
Since $\mathcal T$ is $m_2$-nice,
we have
\[
\|\mathcal M^{(m_2)}_{T,N}
\|
\geq (1-1/m_2)
\mbox{ for all } N \mbox{ large enough}.
\]
Recall also that $\mathcal M^{(m_2)}_{T,N} \leq \mathcal M_{\gamma, N}$
and that 
$\|\mathcal M_{\gamma,N}\|\leq 1+o(1)$. So, in order to prove the statement, it is enough to prove that
\[
1-2/m_2 \leq
\mathcal M_{\gamma,N} (\mathcal D^\star)
\leq 1+1/m_2.
\]
Since $\mathcal M (\mathcal D^\star)\geq 1-1/m_3$ and $m_3 \ge m_2$, this follows from \eqref{e:conv-42-web}.
\end{proof}



We now construct the set $\mathcal Q$ of motions of repelling periodic points as in Proposition \ref{p:l414-reduced}.

\medskip


For every $N_1$ sufficiently large,
every element
in the support
of $\mathbb 1_{\mathcal T^\star} \mathcal M^{(m_3)}_{D_i, N_1}$
corresponds to an $m_3$-good inverse branch of $f:M\times \P^k \to M\times \P^k$ 
of order $N_1$ on the tube $D_i$, which maps $D_i$
to
a 
subset of the tube $T$
whose slice at any $\lambda$ is relatively compact in $T_{|\lambda}$.
In the same way, for every $N_2$ sufficiently large, every element in the support of 
$\mathbb 1_{\mathcal D^\star}\mathcal M^{(m_2)}_{T, N_2}$
corresponds to an $m_2$-good inverse branch of order $N_2$
sending the tube $T$ to a 
subset of $D= \cup_i D_i$.

\medskip

We define the collection $\{g_j\}$ to be the compositions
of such inverse branches, giving inverse branches for $f^{N_1+N_2}$ sending the tube 
$T$ to
a
subset of $T$ whose slice at any $\lambda$ is relatively compact in $T_{|\lambda}$.
Given any such $g_j$, we write $g_j = g_j^{(1)} \circ g_j^{(2)}$, where
$g_j^{(2)}$ is an
inverse branch of $f^{N_2}$ on $T$ with image in $D$ and $g_j^{(1)}$ is an inverse branch of $f^{N_1}$ on some $D_i$ with image in $T$. For every 
$j$, we also set $i=i(j)$  where $i$ is defined
by $g^{(2)}_j (T) \subset D_i$.

By definition, the inverse branch $g_j$ is of the form $g_j(\lam, z) = (\lam, g_j^\lam(z))$, where $\lam \mapsto g_j^\lam(z)$ is holomorphic.
For $\gamma' \in \mathcal T$, set $\mathcal G_j(\gamma')(\lam) := g_j^\lam(\gamma'(\lam))$, so that $\mathcal G_j (\mathcal T) \subseteq \mathcal T^*$.
One can define in a similar way $\mathcal G_j^{(1)} : \mathcal D_i \to \mathcal T^*$ and $\mathcal G_j^{(2)} : \mathcal T \to \mathcal D_i$ and remark that $\G_j = \G_j^{(1)} \circ \G_j^{(2)}$.

\begin{lemma}\label{l:completeness}
The space $(\Supp \M, \dist_{M'})$ is complete.
\end{lemma}
\begin{proof}
Let $(\gamma_n)_{n\in \N}$ be a Cauchy sequence.
As $\Supp \M$ is compact with respect to the local uniform convergence, there exists a subsequence $\gamma_{n_j}$ such that
$\gamma_{n_j}$ converges locally uniformly to a graph $\gamma \in \Supp \M$.
In particular, $\gamma_{n_j}$ converges to $\gamma$ on $M'$ and the Cauchy sequence $(\gamma_n)_{n\in \N}$ admits a cluster value, hence converges.
\end{proof}

\begin{lemma}\label{l:motions}
For $N=N_1+N_2$ large enough, the operator $\G_j$ admits a fixed point $\gamma_j^{\infty} \in \mathcal T \cap \Supp \M$.
Furthermore, $\gamma_j^\infty(\lam)$ is a periodic point of period $N$ of $f_\lam$, for any parameter $\lam \in M$,
and it is repelling for any $\lam \in M'$.
\end{lemma}

\begin{proof}
Since $\gamma \in \Supp \M$, by Lemma \ref{l:criterion-conditions} (6) and Lemma \ref{l:completeness},
it is sufficient to prove that $\G_j$ is a contraction on a closed subset of $\mathcal T$ containing $\gamma$.
Take $\gamma_1, \gamma_2 \in \mathcal T$. 
By Proposition \ref{p:bbd-gen}, we have
\begin{equation}\label{e:contraction}
\dist_{M'}(\G_j(\gamma_1), \G_j(\gamma_2)) \lesssim 
e^{-NA} \dist_{M'}(\gamma_1,\gamma_2),
\end{equation}
where the implicit constant is independent of $N$.
Hence, for all $N$ sufficiently large,
$\G_j$ is 
a contraction on the closed set
$\overline{\mathcal T^*} = \{\gamma' \in \J \colon \dist_\J(\gamma,\gamma') \le \eta^*\}$.
By the Banach fixed point theorem, applied to the complete metric space $\overline{\mathcal T^*}\cap \Supp \M$, the map $\G_j$ admits a  fixed point
$\gamma_j^\infty \in \mathcal T \cap \Supp \M$.
In particular, observe that $\gamma_j^\infty(\lam)$ lies in $J_\lam$ for any $\lam \in M$.
To conclude, it remains to prove that $\gamma_j^\infty(\lam)$ is an $N$-periodic point for every $\lam \in M$ and 
is repelling for every $\lam \in M'$.

Recall that $T=T(\gamma, \eta)$ is the tube associated to the ball $\mathcal T$.
Remark that $f^N \colon g_j(T) \to T$ is onto,
hence $f^N\circ g_j = id_T$.
For all $\lambda \in M'$,
the equality $\gamma_j^\infty(\lam) = \G_j(\gamma_j^\infty)(\lam) = g_j(\gamma_j^\infty(\lam))$
then leads to $f_\lam^N(\gamma_j^\infty(\lam)) = \gamma_j^\infty(\lam)$. This shows that $\gamma_j^\infty$ is indeed the motion of an $N$-periodic point over $M'$.
In order to extend it to a motion over all $M$, it is sufficient to show that $\gamma_j^\infty(\lam)$ is repelling for any $\lam \in M'$.
Indeed, since $\gamma_j^\infty \in \Supp \M$, \cite[Lemma 2.5]{BBD18} then allows us to conclude.

Observe that
\eqref{e:contraction} also 
gives
a positive constant $c$ such that
\[
\dist_{M'}(\F^{qN}(\gamma_1),\gamma_j^\infty) 
\gtrsim 
e^{qNA}
\dist_{M'}(\gamma_1, \gamma_j^\infty),
\]
for every positive integer $q \in \N^*$ and any graph $\gamma_1 \in \mathcal T \left( \G_j^q(\gamma_1), 
c\cdot e^{-qNA}
\right)
$, where the implicit constant is independent of $q$.
This proves that $\gamma_j^\infty(\lam)$ is repelling for any $\lam \in M'$, and concludes the proof.
\end{proof}



We can now conclude the proof of Proposition \ref{p:l414-reduced}. As explained above, this also completes the proof of Proposition \ref{p:l414}.

\begin{proof}[End of the proof of Proposition \ref{p:l414-reduced}]
%
%
%
%
%
We continue to use the notations introduced above. 
Up to possibly increasing $M_1$ and $M_2$ in Lemmas \ref{l:claim1} and \ref{l:claim2}, we can
assume that we can take $N_1=M_1$ and $N_2=M_2$ in the construction before Lemma \ref{l:motions}.
We 
set $n(m_2):= M_1 (m_2)+M_2 (m_2)$, for a fixed choice of sufficiently large $m_2$ and $m_3$ as above. For every $n> n(n_2)$, we
denote by $\mathcal Q_n$ the set of the motions of $n$-repelling point $\sigma_j$ given by Lemma \ref{l:motions}
applied with $N_1 = M_1(m_2)$ and $N_2 = n-M_2 (m_2)$.
Every element $\sigma\in \mathcal Q_n$ is then the motion of
a repelling $n$-periodic point.






Set
\[
\mathcal M_n :=
\sum_{\sigma \in \mathcal Q_n}
e^{\psi (\gamma) + \ldots + \psi (\mathcal F^{n-1} (\gamma))}
\delta_\gamma
=
\sum_j e^{\psi (\sigma_j) + \ldots + \psi (\mathcal F^{n-1} (\sigma_j))}
\delta_{\sigma_j},
\]
where $\mathcal Q_n = \{\sigma_j\}$, and
\[\begin{aligned}
\tilde{\mathcal M}_n :=
\sum_j
\Big( e^{\psi (\mathcal G_j^{(1)} (\gamma_i(j))) 
+ \ldots +
\psi (\mathcal F^{N_1-1} \circ \mathcal G_j^{(1)} (\gamma_{i(j)}))}
& \cdot \frac{\theta (\mathcal G_j^{(1)} (\gamma_{i(j)})) }{\theta (\gamma_{i(j)})}\\
& \cdot e^{\psi (\mathcal G_j^{2)} (\gamma) 
+ \ldots +
\psi (\mathcal F^{N_2-1} \circ \mathcal G_j^{(2)} (\gamma))}
\frac{\theta (\mathcal G_j^{(2)} (\gamma))}{\theta (\gamma)}
\Big)
\delta_{\mathcal G_j (\gamma)},
\end{aligned}\]
where we recall that $\mathcal T= B(\gamma, \eta)$ 
and $\mathcal D_i =B(\gamma_i, \eta_i)$. Observe in particular that
\begin{equation}\label{e:MntildeT}
\tilde {\mathcal M}_n (\mathcal T)
=
\sum_i \mathcal M_{T,N_2}^{(m_2)} (\mathcal D^\star_i)
\cdot 
\mathcal M_{D_i, N_1}^{(m_3)} (\mathcal T^\star).
\end{equation}
By Lemma \ref{l:claim1} and the fact that 
$\sum_i \mathcal M^{(m_2)}_{T, N_2} (\mathcal D_i^\star)=
\mathcal M^{(m_2)}_{T,N_2} (\mathcal D^\star)$, 
we have
\[
(1-4/m_2)
\mathcal M(\mathcal T)
\mathcal M^{(m_2)}_{T,N_2} (\mathcal D^\star)
\leq 
\tilde{\mathcal M}_n (\mathcal T)
\leq
(1+4/m_2)
\mathcal M (\mathcal T)
\mathcal M^{(m_2)}_{T,N_2} (\mathcal D^\star).
\]
By Lemma \ref{l:claim2}, this gives
\begin{equation}\label{e:estimate-tilde}
(1-10/m_2)
\mathcal M(\mathcal T)
\leq 
\tilde{\mathcal M}_n (\mathcal T)
\leq
(1+10/m_2)
\mathcal M (\mathcal T).
\end{equation}
Hence, 
up to taking $m_2$ sufficiently large,
the desired estimate holds if we replace $\mathcal M_n (\mathcal T)$ with
$\tilde {\mathcal M}_n (\mathcal T)$.
It is then enough to compare
$\mathcal M_n (\mathcal T)$ with
$\tilde {\mathcal M}_n (\mathcal T)$.


\medskip

Observe that, for all $i$ and $j$,
the graphs of $\gamma$ and $\mathcal G_j^{(1)}(\gamma_{i(j)})$
belong to $T$ and 
those of $\gamma_{i(j)}$
and $\mathcal G_j^{(2)} (\gamma)$ belong to $D_i$. Hence, 
since the tubes $T$ and $D_i$ are $m_2$-nice
and satisfy \eqref{e:small_enough}
(with $\mathcal A$ replaced by $\mathcal T$ and $\mathcal D_i$, respectively), we have
\[
\Big|
\frac{\theta ( \mathcal G_j^{(1)} (\gamma_{i(j)}) )}{\theta (\gamma)}
-1\Big|\le \frac{2}{m_2}
\mbox{ and }
\Big|
\frac{\theta (\mathcal G_j^{(2)} (\gamma))}{\theta(\gamma_{i(j)})}
-1\Big|\le \frac{2}{m_2}
\mbox{ for all }
j.\]
It follows from these inequalities and Lemma \ref{l:l413}
that 
$\lvert\mathcal M_n (\mathcal T)- \tilde{\mathcal M}_n (\mathcal T)\rvert  \lesssim
\tilde{\mathcal M}_n (\mathcal T)/m_2$. 
This, together with \eqref{e:estimate-tilde}, concludes the proof.
%
%
%
%
%
\end{proof}

\subsection{Proof of Theorem \ref{t:main}}\label{ss:proof-main}




We can now conclude the proof of Theorem \ref{t:main}. Recall that we are assuming that $P(\phi)=0$ and that 
we denote by $\mathcal M$ the web $\mathcal M_{\lam_0, \mu_\phi}$, whose construction is detailed in Section \ref{ss:equilibrium}.



\medskip

We first build a  family
of measurable partitions of $(\J, \M)$ whose diameter converges to $0$.

\begin{lemma}\label{l:partition}
There exists a sequence $(\mathfrak U_i)_{i\in \N}$ of finite families of disjoint open sets $\mathfrak U_i = \{\mathcal U_{i,j}\}_{1\leq j \leq J_i}$ satisfying the following properties:
\begin{enumerate}
\item[{\bf (U1)}] for all $i \in \N$, we have $\mathcal M (\cup_{1\leq j \leq J_i} \mathcal U_{i,j})=1$;
\item[{\bf (U2)}] $\sup_{i\in\N} (i \cdot \max_{1\leq j \leq J_i}\diam_{\J} \mathcal U_{i,j})\leq 1$;
\item[{\bf (U3)}] for all $i\geq 2$ and $1\leq j \leq J_i$
there exists $1\leq j'\leq J_{i-1}$ such that
$\mathcal U_{i,j} \subset \mathcal U_{i-1,j'}$.
\end{enumerate}
\end{lemma}

\begin{proof}
We prove the lemma by induction on $i\in \mathbb N$.
We set $J_0 = 1$ and $\mathcal U_{0,1} = \J$.
The family $\mathfrak U_0 :=\{U_{0,1}\}$
satisfies conditions {\bf (U1)} and {\bf (U2)}, and condition 
{\bf (U3)} is empty in this case.

\medskip

We now do the induction step. Assume that we have built, for some $i \in \N$, a finite family $\mathfrak U_i = \{\mathcal U_{i,j}\}_{1\le j \le J_i}$ satisfying \textbf{(U1)}, \textbf{(U2)}, and \textbf{(U3)}.
Consider the interval $I:= (\frac{1}{2(i+2)}, \frac{1}{2(i+1)})$.
As $\Supp \M \subseteq  \cup_{\gamma \in \Supp \mathcal M} \mathcal B_{\J}(\gamma, \frac{1}{2(i+2)})$ 
and $\Supp \mathcal M$ is compact, there exists $m_i \in \N$ and a collection $\{\gamma_{i,m}\}_{1\leq m\leq m_i} \subset \Supp \M$
such that, for any $\epsilon \in I$, $\Supp \M \subseteq \cup_{m=1}^{m_i} \mathcal B_{\J}(\gamma_{i,m}, \epsilon)$.
Set
$B_{i,m}^\epsilon := \mathcal B_{\J}(\gamma_{i,m},\epsilon)$ and $D_{i,m}^\epsilon := \partial \mathcal B_{i,m}^\epsilon$.

\medskip

\noindent \textbf{Claim.} The set $A_i:=\{\epsilon \in I \::\:\M(\cup_{1\le m\le m_i} D_{i,m}^\epsilon) > 0\}$ is countable.

\begin{proof}
As the family $\{D_{i,m}^\epsilon\}_{1\leq m \leq m_i}$ is finite, 
it suffices to prove that 
$A_{i,m} := \{\epsilon \in I \::\: \M(D_{i,m}^\epsilon) > 0\}$
is countable
for every 
$1\le m \le m_i$.
Assume this is not true for some $A_{i,\bar m}$.

For $\alpha \in \mathbb N$, set $A_{i,\bar m}^\alpha = \{\epsilon \in I \::\: \M(D_{i,\bar m}^\epsilon) > \frac{1}{\alpha +1}\}$. Observe that
$A_{i,\bar m} = \cup_{\alpha \in \N} A_{i,\bar m}^\alpha$.
Hence, there exists $\alpha \in \N$
such that $\Card(A_{i,\bar m}^\alpha) = +\infty$. As $D_{i,\bar m}^{\epsilon}\cap D_{i,\bar m}^{\epsilon'}=\emptyset$
for all $\epsilon \neq \epsilon' \in I$, we have
\begin{equation*}
1 \ge \M(\cup_{\epsilon \in A_{i,\bar m}^\alpha}D_{i,\bar m}^\epsilon) = \sum_{\epsilon \in A_{i,\bar m}^\alpha} \M(D_{i,\bar m}^\epsilon) \ge \frac{\Card(A_{i,\bar m}^\alpha)}{\alpha+1} = +\infty.
\end{equation*}
This gives a contradiction, and the proof of the claim is complete.\end{proof}

Fix now $\epsilon \in I$ such that
$\M(D_{i,m}^\epsilon) = 0$
for all $1\le m \le m_i$.
Denote the connected components of $\cup_{j=1}^{J_i} \mathcal U_{i,j} \cap B_{i,1}^\epsilon$ by $\mathcal U_{i+1,1}, \dots, \mathcal U_{i+1, l_1}$. By induction, 
for $1<m\le m_i$, define also $\mathcal U_{i+1, l_{m-1}+1}, \dots,\mathcal U_{i+1,l_m}$ as the connected components of
$$\bigcup_{j=1}^{J_i} \mathcal U_{i,j} \cap \left( B_{i,m}^\epsilon \setminus \bigcup_{l=1}^{l_{m-1}} \overline{\mathcal{U}_{i+1,l}} \right).$$
Set $J_{i+1} := l_{m_i}$. 
By definition, the sets $\mathcal U_{i+1,l}$ are open and pairwise disjoint.
We claim that
\[
\bigcup_{l=1}^{J_{i+1}} \mathcal U_{i+1,l} \supset \left(
\bigcup_{m=1}^{m_i} B_{i,m}^\epsilon \setminus \bigcup_{l=1}^{J_{i+1}} \partial \mathcal U_{i+1,l}\right)
\cap \bigcup_{j=1}^{J_i} \mathcal U_{i,j}.
\]

Fix $m\leq m_i$ 
and take $\gamma \in B_{i,m}^\epsilon\setminus
\cup_{1\le l\le J_{i+1}}\partial \mathcal U_{i+1,l}$,
for some $m \le m_i$. We can assume that $\gamma$ does not satisfy this property for all $m'<m$.
By the definition of the sets $\mathcal U_{i+1,l}$,
this implies that $\gamma \notin \mathcal U_{i+1,l}$, for every $l\le l_{m-1}$.
Hence $\gamma \in B_{i,m}^\epsilon \setminus \bigcup_{l=1}^{l_m-1} \overline{\mathcal U_{i+1,l}}$.
This proves the claim.

Remark that $\bigcup_{l=1}^{J_{i+1}} \partial \mathcal U_{i+1,l} \subset \bigcup_{m=1}^{m_i} D_{i,m}^\epsilon$.
As $\mathfrak U_i$ satisfies 
\textbf{(U1)},
the choice of $\epsilon$
yields
\begin{equation*}
\M\left(\bigcup_{l=1}^{J_{i+1}} \mathcal U_{i+1,l} \right) 
\ge \M \left( \bigcup_{m=1}^{m_i} B_{i,m}^\epsilon \setminus \bigcup_{l=1}^{J_{i+1}} \partial \mathcal U_{i+1,l} \right)
\ge \M\left(\bigcup_{m=1}^{m_i} B_{i,m}^\epsilon\right) - \M\left(\bigcup_{m=1}^{m_i} D_{i,m}^\epsilon \right)
 = 1.
\end{equation*}

This proves that $\mathfrak U_{i+1}$ satisfies \textbf{(U1)}. The property \textbf{(U2)} comes from 
the definition of the $B^\epsilon_{i,m}$'s
and the fact that, by construction, 
for every $l$ there exists $m(l)$ such that  
$\mathcal U_{i+1,l} \subset B_{i,m(l)}^\epsilon$.
Finally, the fact that $\mathcal U_{i+1,l}$ is a connected subset of $\cup_{j=1}^{J_i} \mathcal U_{i,j}$ gives \textbf{(U3)}. The proof is complete.
\end{proof}



\medskip

For every $n\in \mathbb N$, define 
$i_n := \max \{ m\leq n \colon n \geq n(m, \mathfrak U_m)\}$, where $n(m, \mathfrak U_m)$ is given by Proposition \ref{p:l414}.
For every  $\alpha >0$, set $m_\alpha := \lfloor \alpha+1 \rfloor$. Then, for any $n \ge \max(m,n(m,\mathcal U_m))$, we have $i_n \ge m_\alpha > \alpha$.
In particular,
we have $i_n\to \infty$ as $n\to \infty$.
 
For every $n\in \mathbb N$, we apply Proposition \ref{p:l414} with ${\mathfrak U}_{i_n}$
instead of $\mathfrak U$.
This gives a collection $\mathcal P_{\phi,n}\subset \cup_{1\leq j \leq J_{i_n}} \mathcal U_{i_n, j}$ of motions of repelling $n$-periodic points, satisfying the properties in that statement.
We define
\begin{equation}\label{eq:M'n}
\mathcal M'_n :=
\sum_{\sigma \in \mathcal P_{\phi,n}}
e^{\psi (\sigma) + \ldots + \psi (\mathcal F^{n-1} \sigma)}\delta_\sigma.
\end{equation}

\begin{lemma}\label{l:limit-1}
Any limit $\mathcal M'$ of the sequence $\{\mathcal M'_n\}_{n \in \mathbb N}$
has mass 1.
\end{lemma}

\begin{proof}
 By definition, for every $n\in \mathbb N$ we have
\[
\lVert \mathcal M'_n \rVert =
\sum_{\sigma \in \mathcal P_{\phi,n}}
e^{\psi (\sigma) + \ldots +\psi (\mathcal F^{n-1} \sigma)}
= \sum_{j=1}^{J_{i_n}} \sum_{\sigma \in \mathcal P_{\phi,n}\cap \mathcal U_{i_n,j}}
e^{\psi (\sigma) + \ldots +\psi (\mathcal F^{n-1} \sigma)}.
\]
Since the family $\mathfrak U_n$ is pairwise disjoint, Proposition \ref{p:l414}  yields
\[
\left(1-\frac{1}{i_n}\right) \M_\phi\left(\bigcup_{j=1}^{J_{i_n}}\mathcal U_{i_n,j}\right) \le \lVert \M_n' \rVert \le \left(1+\frac{1}{i_n}\right) \M_\phi\left(\bigcup_{j=1}^{J_{i_n}} \mathcal U_{i_n,j}\right).
\]
We conclude by Property \textbf{(U1)}
and letting $n\to \infty$.
\end{proof}


\begin{lemma}\label{l:liminf}
For all $i^\star \in \mathbb N$ and $1\leq j^\star \leq J_{i^\star}$ we have
\[
\liminf_{n\to \infty}
\mathcal M'_n (\mathcal U_{i^\star, j^\star})
\geq \mathcal M (\mathcal U_{i^\star, j^\star}).
\]
\end{lemma}

\begin{proof}
Fix $i^\star, j^\star$ as in the statement and 
$\epsilon >0$. It is enough to prove that, for all $n$ sufficiently large, we have
\[
\mathcal M'_n (\mathcal U_{i^\star, j^\star})
\geq \mathcal M  (\mathcal U_{i^\star, j^\star})
-\epsilon.
\]
It is enough to consider only those $n$ for which $i_n > i^\star$. For all such $n$, by 
{\bf (U1)} 
and pairwise disjointness,
we have
\[
\mathcal M  (\mathcal U_{i^\star, j^\star})
=
\sum_{j'} \mathcal M (\mathcal U_{i_n, j'})
\quad \mbox{ and}
\quad
\mathcal M'_n  (\mathcal U_{i^\star, j^\star})
\ge 
\sum_{j'} \mathcal M'_n (\mathcal U_{i_n, j'}),
\]
where the two sums are
over the $j'$ such that $U_{i_n,j'}\subset U_{i^\star, j^\star}$.
Using the same convention for the sums,
by the definition of $\mathcal M'_n$ and Proposition \ref{p:l414}, we have 
\[
\mathcal M  (\mathcal U_{i^\star, j^\star}) -
\mathcal M'_n  (\mathcal U_{i^\star, j^\star})
\leq
\sum_{j'}
\mathcal M  (\mathcal U_{i_n, j'}) -
\mathcal M'_n  (\mathcal U_{i_n, j'})
\leq
i_n^{-1}
\sum_{j'}
\M(\mathcal U_{i_n, j'})
=
i_n^{-1} \M(\mathcal U_{i^*,j^*}).
\]
The assertion follows.
\end{proof}

\begin{proof}[End of the proof of Theorem \ref{t:main}]
We show that the sequence $\{\M'_n\}_{n\in \N}$ as in \eqref{eq:M'n}
converges to $\M$.
Let $\mathcal A \subset \J$ be a closed set. Consider $\mathcal A' := \bigcap_{n\in \N} \bigcup_{\mathcal U \in \mathfrak U_{i_n}^{\mathcal A}}\mathcal U$, where $\mathfrak U_i^{\mathcal A} := \{\mathcal U\in \mathfrak U_i \::\:
\mathcal U \cap \mathcal A \neq \emptyset \}$. By \textbf{(U1)}, for every $n\in \mathbb N$ we have
\[
\M\left(
\bigcup_{\mathcal U 
\in \mathfrak U_{i_n}^{\mathcal A}} \mathcal U\right) = 1 - \M\left(\bigcup_{\mathcal U 
\in \mathfrak U_{i_n} \setminus \mathfrak U_{i_n}^{\mathcal A}
}\mathcal U\right) \ge 1 - \M(\mathcal A^c)=
\mathcal M (\mathcal A).
\]
Remark that property \textbf{(U3)} ensures that $\mathcal A'$ is a countable non-increasing intersection.
Letting $n\to \infty$, we deduce that $\M(\mathcal A') \ge \M(\mathcal A)$.

\medskip

\noindent
\textbf{Claim.} We have $\mathcal A'\subseteq \mathcal A$. In particular, we have $\mathcal M(\mathcal A') = \mathcal M( \mathcal A)$.

\begin{proof}
Since $\mathcal M(\mathcal A')\leq\mathcal M(A)$, it is enough to show that $\mathcal A'\subseteq \mathcal A$.
Take $\gamma \in \mathcal A'$.
For any $n \in \mathbb N$, there exist some open set $\mathcal U \in \mathfrak U_{i_n}$ and $\gamma_n \in \mathcal J$ with
$\gamma \in \mathcal U$ and $\gamma_n \in \mathcal U \cap \mathcal A$.
By property \textbf{(U2)}, we have $\dist_{\J}(\gamma,\gamma_n) < 1/{i_n}$. We deduce that $\gamma = \displaystyle \lim_{n\to \infty}\gamma_n \in \overline{\mathcal A} = \mathcal A$.
\end{proof}

Let now $\M'$ be any limit of the sequence $\{\M_n'\}_{n\in \N}$. Lemma \ref{l:liminf} and the Claim above give
that
\[
\M(\mathcal A) = \M(\mathcal A') = \lim_{n\to \infty} \sum_{\mathcal U \in \mathfrak U_{i_n}^{\mathcal A}}\M(\mathcal U)
\le \lim_{n\to \infty} \sum_{\mathcal U \in \mathfrak U_{i_n}^{\mathcal A}}\M'(\mathcal U)
= \M'(\mathcal A') \le \M'(\mathcal A).
\]
As the closed set $\mathcal A$ was chosen arbitrarily,
Lemma \ref{l:limit-1} and the fact that $\|\mathcal M\|=1$
imply that
that $\M = \M'$.
This shows that $\mathcal M'_n\to \mathcal M$ as $n\to \infty$, and completes the proof.
\end{proof}

\bibliographystyle{alpha}

\end{document}